\documentclass[reqno]{amsart} 
\usepackage{xcolor}

\definecolor{darkgreen}{rgb}{0,0.5,0}
 
\usepackage{cite}
\usepackage{amsmath,amssymb,amsfonts, amsthm}
\usepackage{algorithmic}
\usepackage{algorithm}
\usepackage{graphicx}
\usepackage{subcaption}
\usepackage{tikz}
\usepackage[colorlinks=true, allcolors=blue, linktocpage=true]{hyperref}
\usepackage{textcomp}
\usepackage{mathtools}
\usepackage{enumerate}
\usepackage{xfrac}
\usepackage{upgreek}

\definecolor{BBlue}{cmyk}{.98,0.10,0,.25}
\definecolor{LightBlue}{rgb}{0.2,0.6,1}
\definecolor{DarkBlue}{rgb}{0,0,.5}
\definecolor{DarkGreen}{rgb}{0,0.4,.2}
\definecolor{LightGreen}{rgb}{0.6,1,.4}

\def\BibTeX{{\rm B\kern-.05em{\sc i\kern-.025em b}\kern-.08em
    T\kern-.1667em\lower.7ex\hbox{E}\kern-.125emX}}

\newcommand{\real}{{\mathbb{R}}}
\newcommand{\hodclf}{\textsf{g-dclf}}
\newcommand{\adc}{\textsf{adc}}

\newcommand{\interior}{\operatorname{int}}
\newcommand{\boldnu}{\boldsymbol{\nu}}
\newcommand{\boldu}{\mathbf{u}}
\newcommand{\initimetimev}{k_0}
\newcommand\subscr[2]{#1_{\textup{#2}}}
\newcommand\ld{\subscr{\ell}{decr}}
\DeclareMathOperator{\vect}{vec}

\newtheorem{theorem}{Theorem}[section]

\newtheorem{lemma}[theorem]{Lemma}
\newtheorem{definition}[theorem]{Definition}

\newtheorem{problem}[theorem]{Problem}

\DeclareMathOperator{\trace}{trace}

\makeatletter
\renewcommand\subsection{\@startsection{subsection}{2}%
  \z@{-.5\linespacing\@plus-.7\linespacing}{.5\linespacing}%
  {\normalfont\scshape}}
\renewcommand\subsubsection{\@startsection{subsubsection}{3}%
  \z@{.5\linespacing\@plus.7\linespacing}{-.5em}%
  {\normalfont\scshape}}
\makeatother
    
\begin{document}
\title[]{Flexible-step MPC for Switched Linear Systems with No Quadratic Common Lyapunov Function}

\author[Annika F\"urnsinn]{Annika F\"urnsinn}
\address{Department of Mathematics and Statistics\\ Queen's University, Kingston, ON K7L 3N6, Canada}
\email{19af7@queensu.ca}
\author[Christian Ebenbauer ]{Christian Ebenbauer }
\address{Chair of Intelligent Control Systems\\
RWTH Aachen University, 52062 Aachen, Germany}
\email{christian.ebenbauer@ic.rwth-aachen.de}
\author[Bahman Gharesifard ]{Bahman Gharesifard }
\address{Department of Mathematics and Statistics\\ Queen's University, Kingston, ON K7L 3N6, Canada}
\email{bahman.gharesifard@queensu.ca}

\begin{abstract}
In this paper, we develop a systematic method for constructing a generalized discrete-time control Lyapunov function for the flexible-step Model Predictive Control (MPC) scheme, recently introduced in~\cite{AF-CE-BG:22}, when restricted to the class of linear systems. Specifically, we show that a set of Linear Matrix Inequalities (LMIs) can be used for this purpose, demonstrating its tractability. 
The main consequence of this LMI formulation is that, when combined with flexible-step MPC, we can effectively stabilize switched control systems, for which no quadratic common Lyapunov function exists. 
\end{abstract}

\maketitle


\section{Introduction}
Model predictive control (MPC) is a widely utilized method to approximate the solution of an infinite-horizon optimal control problem, where the underlying dynamics can be of nonlinear or linear nature. In this paper, we focus on the latter. The method itself consists of considering a sequence of finite-horizon optimal control problems implemented in a receding horizon fashion. The performance and stability properties of MPC schemes heavily rely on terminal conditions~\cite{DM-JR-CR-PS:00, DM:14,JR-DM-MD:17, LG-JP:17}; in fact, in the standard scheme, the optimal value function is chosen as the Lyapunov function and stability is ensured by enforcing a one-step descent of this optimal value function.

In our recent work~\cite{AF-CE-BG:22}, we proposed the flexible-step MPC scheme, which notably allows for implementing a flexible number of control inputs in each iteration in a computationally attractive manner. Here, with flexible number we mean that in each iteration of the scheme the number of implemented elements of the corresponding optimal input sequence is not constant, e.g. one (standard MPC), but instead flexible (variable). In order to provide stability guarantees, we introduced the notion of generalized discrete-time control Lyapunov functions (\hodclf s). As opposed to the strict descent required of classical Lyapunov functions, this notion only requires a descent of the averaged values of these functions over a maximal allowable implementation window.  We incorporate this average decrease condition as a constraint in our MPC formulation. As demonstrated in our prior work~\cite{AF-CE-BG:22, AF-CE-BG:23}, the flexible-step implementation of our scheme has many advantages; for example, it can be used for stabilizing nonholonomic systems, where it is known that standard MPC may lack asymptotic convergence~\cite{MM-KW:17}.

A characteristic of Lyapunov theory is its reliance on a suitable Lyapunov function.   
In spite of its powerful properties, the same is true for flexible-step MPC; here, we need to find a suitable \hodclf. The purpose of the current paper is to provide a systematic way of constructing such functions for the class of linear control systems. In particular, we show that a simple quadratic function can be selected as the \hodclf. In correspondence with the related non-monotonic Lyapunov functions~\cite{AA-PP:08}, we can cast the average decrease constraint (\adc) as an LMI, which can then be efficiently solved. More importantly, we show that by choosing a large enough time step, this LMI always admits a solution. 

One advantage of this LMI characterization is that it can be utilized without further knowledge about the cost function. This is due to the fact that in flexible-step MPC, the two tasks of optimization and stabilization are decoupled, i.e. the \hodclf~can generally be a different function than the optimal cost function. In standard MPC on the other hand, the tasks of optimization and stabilization go hand in hand. The connection is particularly apparent, when working with a linear system and a quadratic cost function. In this case, the terminal cost consists of the solution to a Riccati or Lyapunov equation. For non-quadratic cost functions, there is no straightforward procedure for the design in standard MPC, however, as mentioned earlier, the LMI approach can still be utilized for flexible-step MPC.

More importantly, the primary advantage of this LMI characterization is that it allows us to use the flexible-step MPC scheme in the context of switched linear systems. It is well-known that there are globally asymptotically stable switched systems for which no classical quadratic common Lyapunov function exists~\cite{DL:03, JH-DL-DA-ES:05}. 
In this paper, we provide an example of such a system, for which, in contrast, 
a quadratic common \hodclf~exists. In particular, we find this function by coupling the LMI conditions for the individual linear components.  Our result demonstrates that similar to nonholonomic systems~\cite{AF-CE-BG:22}, flexible-step MPC is also a useful tool for stabilizing switched linear systems.

The remainder of the paper is organized as follows. We give a brief summary of the flexible-step MPC scheme in Section~\ref{sec:flexMPC}. Section~\ref{sec:avenuelmi} contains our main result, where we explain the translation of the \adc~constraint into an LMI and discuss a numerical MPC example with a non-quadratic cost function. In Section~\ref{sec:commonLF}, we first apply the LMI approach and then flexible-step MPC to a linear switched control system and present the numerical results. We discuss our future directions in Section~\ref{sec:conclusion}.

\subsection{Notation}\label{sec:prelim}
We denote the set of non-negative (positive) integers by $\mathbb{N}$ ($\mathbb{N}_{>0}$), the set of (non-negative) reals by $\real$ ($\mathbb{R}_{\geq 0}$), and the interior of a subset $S \subseteq \mathbb{R}^n$ by $\interior S$. The set of real matrices of the size $p\times q$ is denoted by $\real^{p,q}$. We make use of boldface when considering a sequence of finite vectors, e.g. $\mathbf{u}= [u_{0},\dots,u_{N-1}] \in \real^{p,N}$, we refer to $j$th component by $u_j$ and to the subsequence going from component $i$ to $j$ by $\mathbf{u}_{[i:j]}$. Similarly, $\mathbf{U}_{[0:N-1]} \subseteq \mathbb{R}^{p,N}$ denotes a set of (ordered) sequences of vectors with components zero to $N-1$. When such a set depends on the initial state $x$, it is expressed by $\mathbf{U}_{[0:N-1]}(x)$. As usual, $\boldu^{*}$ denotes the solution to the optimal control problem solved in the iteration of an MPC scheme, the symbol $\boldu^{*-}$ denotes the optimal control sequence of the \emph{previous} iteration. For later use, we recall that a function $V: \mathbb{R}^n\times\mathbb{R}^{p, q} \to \mathbb{R}$ is called positive definite if $V(x,\boldu) = 0$ is equivalent to $ (x,\boldu) = (0,\mathbf{0})$ and $ V(x,\boldu) > 0$ for all $(x,\boldu) \in \mathbb{R}^n\times\mathbb{R}^{p, q}\setminus \{(0, \mathbf{0})\}$. Furthermore, a function $\alpha: \mathbb{R}^n\times\mathbb{R}^{p, q} \to \mathbb{R}$ is called radially unbounded if $\|(x,\boldu)\| \to \infty$ implies $\alpha(x,\boldu) \to \infty$, where $\|\cdot\|$ is the Euclidean norm. Note that $x$ is a vector and $\boldu$ is a matrix here, so with $\|(x,\boldu)\|$ we implicitly refer to the norm of $[x^T,\vect(\boldu)^T]$, where $\vect(\boldu)$ is the usual vectorization of a matrix into a column vector. 

In the setting of MPC, it is important to distinguish between predictions and the actual implementations at a given time index $k \in \mathbb{N}$. 
In particular, we use $x^k, u^k$ to refer to \emph{predictions}, whereas we use $x(k), u(k)$ to refer to the \emph{actual} states and \emph{implemented} inputs, respectively. 


\section{Flexible-step MPC}\label{sec:flexMPC}
To begin with, let us consider nonlinear discrete-time control systems of the form
\begin{align}\label{eqn:control_system}
    x^{k+1} &= f(x^k,u^k),
\end{align} 
where $ k \in \mathbb{N} $ denotes the time index, $x^k \in X \subseteq \mathbb{R}^n$ is the state with the initial state $x^0 \in X$ and $u^k \in U \subseteq \mathbb{R}^p$ is an input. The state and input constraints satisfy $0 \in \interior X$, $0 \in \interior U$ and $f: X \times U \to \mathbb{R}^n$ is continuous with $f(0,0)=0$. The core of the flexible-step MPC scheme consists of \hodclf s, which are defined as follows~\cite{AF-CE-BG:22}:
\begin{definition}[Set of Feasible Controls]
For $x\in X$, we define a set of feasible controls as 
\begin{align}
    \mathbf{U}_{[0:N-1]}(x):=\{\boldu_{[0:N-1]} \in \mathbb{R}^{p,N}: &\text{ with $x^0=x$}, u_j \in U, x^{j+1} = f(x^j,u_j) \in X,\label{eqn:setfeasc} \\
    &\hphantom{6cm} j = 0, \dots, N-1\}.\nonumber 
\end{align}
An infinite sequence of control inputs is called feasible when equation~\eqref{eqn:setfeasc} is fulfilled for all times $j \in \mathbb{N}$. 
\end{definition}
\begin{definition}[\hodclf]\label{def:ho-dclf}
Consider the control system~\eqref{eqn:control_system}. Let $m \in \mathbb{N}_{>0}$ and $q \in \mathbb{N}$. We call $V: \mathbb{R}^n \times \mathbb{R}^{p, q} \to \mathbb{R} $ a generalized discrete-time control Lyapunov function of order $m$ (\hodclf) for system~\eqref{eqn:control_system} if $V$ is continuous, positive definite and additionally:

\underline{When $q = 0$:}  
\begin{enumerate}[i)]
\setcounter{enumi}{0}
    \item there exists a continuous, radially unbounded and positive definite function ${\alpha: \mathbb{R}^n \to \mathbb{R}}$ such that
    for any $x^0 \in$ $\mathbb{R}^n $ we have
$
    V(x^0) - \alpha(x^0) \geq 0;
$
\item for any $x^0 \in \mathbb{R}^n$ there exists $\boldnu_{[0:m-1]} \in \mathbf{U}_{[0:m-1]}(x^0)$, which steers $x^0$ to some $x^{m}$, such that
\begin{equation}\label{eqn:adc_q0}
    \hspace{-0.1cm}\tfrac{1}{m}(\sigma_m V(x^{m}) + \dots +\sigma_1 V(x^1)) -V(x^0) \leq -\alpha(x^0),
\end{equation}
where $\sigma_m, \ldots, \sigma_1 \in \real_{\geq 0}$ and 
\begin{equation}\label{eqn:sigma_ave}
    \tfrac{1}{m}(\sigma_m + \sigma_{m-1} + \ldots + \sigma_1) -1 \geq 0.
\end{equation}
\end{enumerate} 

\underline{When $q \neq 0$:}
\begin{enumerate}[i')]
\setcounter{enumi}{0}
    \item there exists a continuous, radially unbounded and positive definite function ${\alpha: \mathbb{R}^n \times \mathbb{R}^{p, q}\to \mathbb{R}}$ such that for any $(x^0,\boldu_{[0:q-1]}) \in$ ${\mathbb{R}^n \times \mathbf{U}_{[0:q-1]}(x^0)}$ we have
\begin{align}\label{eqn:cond_V_alpha}
    V(x^0, \boldu_{[0:q-1]}) - \alpha(x^0,\boldu_{[0:q-1]}) \geq 0;
\end{align}
\item for any $(x^0,\boldu_{[0:q-1]}) \in \mathbb{R}^n \times \mathbf{U}_{[0:q-1]}(x^0)$ there exists $\boldsymbol{\nu}_{[0:q+m-1]}\in \mathbb{R}^{p,q+m}$
with $\boldnu_{[l:q+l-1]} \in \mathbf{U}_{[0:q-1]}(x^l)$ for every $l \in \{0,1,\dots,m\}$, which steers $x^0$ to some $x^{m}$, such that
\begin{align}\label{eqn:dclf_orderm}
    &\tfrac{1}{m}(\sigma_m V(x^{m},\boldsymbol{\nu}_{[m:q+m-1]}) + \dots +\sigma_1 V(x^1,\boldsymbol{\nu}_{[1:q]}))-V(x^0, \boldu_{[0:q-1]}) \leq -\alpha(x^0,\boldu_{[0:q-1]}),
\end{align}
where $\sigma_m, \ldots, \sigma_1 \in \real_{\geq 0}$ satisfy~\eqref{eqn:sigma_ave}.
\end{enumerate}
\end{definition}

The interpretation of Definition~\ref{def:ho-dclf} is that the sequence of Lyapunov function
values decreases on average every $m$ steps, and not at every step as in the classical Lyapunov function. This is why we will refer to~\eqref{eqn:dclf_orderm} (or~\eqref{eqn:adc_q0} in the case $q=0$) as the average decrease
 constraint (\adc).
This notion is key in the flexible-step MPC scheme proposed in~\cite{AF-CE-BG:22}, which 
aims to solve the following optimal control problem:
\begin{align}\label{eq:display}
    \begin{aligned}
    &\hspace{-0.4cm}\min \sum_{j=0}^{N_p -1} f_0(x^j,u^j) + \phi(x^{N_p})\\
    \ &\hspace{-0.3cm}\text{s.t. }  x^{j+1} = f(x^j,u^j), \ j = 0, \dots, N_p-1, \\ 
    &\hspace{-0.3cm}\hphantom{\text{s.t. }} x^0 = x, \\ 
     &\hspace{-0.3cm}\hphantom{\text{s.t. }} u^j \in U,\,x^j \in X, \, x^{N_p} \in X^{N_p}, \ j = 0, \dots, N_p-1  \\ 
    &\hspace{-0.3cm}\hphantom{\text{s.t. }} [u^l, \dots, u^{l+q-1}] \in \mathbf{U}_{[0:q-1]}(x^l) \text{ for } l=0,1,\dots,m\\
    &\hspace{-0.3cm}\hphantom{\text{s.t. }}\tfrac{1}{m}\sum_{j=1}^{m}\sigma_{j}V(x^{j}, [u^j, \dots,u^{j-1+q}]) - V(x^{0},\boldu^{*-}_{[0:q-1]}) \leq -\alpha(x^0,\boldu^{*-}_{[0:q-1]}),
    \end{aligned}
\end{align}
where $N_p \in \mathbb{N}, q \leq N_p$ with $q\in \mathbb{N}, m \leq N_p$ with $m \in \mathbb{N}_{>0}, X^{N_p} \subseteq X \subseteq \mathbb{R}^n$ with $0 \in \interior X^{N_p}$, $U \subseteq \mathbb{R}^p$ with $0 \in \interior U$, $\sigma_m, \ldots, \sigma_1 \in \real_{\geq 0} $ satisfying~\eqref{eqn:sigma_ave}, $\phi: \mathbb{R}^n \to \mathbb{R}$ is a positive semi-definite function, the function $f_0: \mathbb{R}^n\times\mathbb{R}^p\to \mathbb{R}$ is positive definite and $V: \mathbb{R}^n\times\mathbb{R}^{p, q}\to \mathbb{R}$ is a \hodclf~of order $m$.
Note that condition~\eqref{eqn:dclf_orderm} was added as a constraint to the above optimal control problem.  
The flexible-step MPC scheme is given in Algorithm~\ref{algo:ho-dclf}. We find it useful to explain the idea behind our implementation with an example, illustrated in Figure~\ref{fig:expl_new_mpcimpl}. 
 \begin{algorithm}[ht]
 \caption{Flexible-step MPC scheme}
 \label{algo:ho-dclf}
 \begin{algorithmic}[1]
 \STATE set $k=\initimetimev$, measure the initial state $x(k_0)$ and choose an arbitrary $\boldu^{*-}_{[0:q-1]} \in \mathbf{U}_{[0:q-1]}(x(k_0))$
 \STATE measure the current state $x(k)$ of~\eqref{eqn:control_system}
 \STATE solve problem~\eqref{eq:display} with $x=x(k)$ and obtain the optimal input ${\mathbf{u}^{*}_{[0:N-1]}}$, where $N=\max\{q+m,N_p\}$
 \STATE choose an index $1\leq\ld\leq m$ for which $V(x^{*\ld},\mathbf{u}^*_{[\ld:q+\ld-1]}) - V(x,\boldu^{*-}_{[0:q-1]}) \leq -
     \alpha(x,\boldu^{*-}_{[0:q-1]})$ 
 \STATE implement $\boldu^*_{[0:\ld-1]}=:[c^{k}_{mpc},\dots, c^{k+\ell_{decr}-1}_{mpc}]$ and redefine $\boldu^{*-}_{[0:q-1]}:=\boldu^{*}_{[\ld:q+\ld-1]}$ 
 \STATE increase $k:=k+\ld$ and go to 2
 \end{algorithmic}
 \end{algorithm}  

\begin{figure*}
\begin{subfigure}[b]{0.242\textwidth}
\begin{tikzpicture}[scale=0.485][x=0.5cm, y=0.5cm,domain=0:11,smooth]
   \draw [color=gray!30]  [step=5mm] (-0.3,-0.5) grid (8.5,12);
   \draw[->,thick] (0,0) -- (8,0) node[right] {$k$};
   \draw[->,thick] (0,0) -- (0,11) node[above] {$V$};
   \foreach \c in {1,2,...,7}{
     \draw (\c,-.1) -- (\c,.1) ;
   }
    \filldraw[LightGreen] (0,9) circle (8pt);
    \filldraw[DarkBlue] (0,9) circle (3pt) node[anchor=west] {$V(x^0)$}; 
    \draw[DarkBlue] (0,9) -- (1,8);
    \filldraw[DarkBlue] (1,8) circle (3pt) node[anchor=west] {$V(x^1)$};
    \draw[DarkBlue] (1,8) -- (2,6);
    \filldraw[DarkBlue] (2,6) circle (3pt) node[anchor=west] {$V(x^2)$};
    \draw[DarkBlue,dashed] (2,6) -- (3,7.5);
    \filldraw[DarkBlue] (3,7.5) circle (3pt) node[anchor=west] {$V(x^3)$};
    \draw[DarkBlue,dashed] (3,7.5) -- (4,10);
    \filldraw[DarkBlue] (4,10) circle (3pt) node[anchor=west] {$V(x^4)$};
\end{tikzpicture}
\caption{Optimization at ${k=0}$ \hphantom{platzhalter}}
\label{fig:expl_new_mpcimpl1}
\end{subfigure}
\begin{subfigure}[b]{0.242\textwidth}
\begin{tikzpicture}[scale=0.485][x=0.5cm, y=0.5cm,domain=0:11,smooth]
   \draw [color=gray!30]  [step=5mm] (-0.3,-0.5) grid (8.5,12);
   \draw[->,thick] (0,0) -- (8,0) node[right] {$k$};
   \draw[->,thick] (0,0) -- (0,11) node[above] {$V$};
   \foreach \c in {1,2,...,7}{
     \draw (\c,-.1) -- (\c,.1) ;
   }
    \filldraw[gray!40] (0,9) circle (8pt);
    \filldraw[black!70] (0,9) circle (3pt);
    \draw[black!50] (0,9) -- (1,8);
    \filldraw[black!70] (1,8) circle (3pt);
    \draw[black!50] (1,8) -- (2,6);
    \filldraw[LightGreen] (2,6) circle (8pt);
    \filldraw[LightBlue] (2,6) circle (3pt) node[anchor=south] {$V(x^0)$};
    \draw[black!50,dashed] (2,6) -- (3,7.5);
    \filldraw[black!70] (3,7.5) circle (3pt);
    \draw[black!50,dashed] (3,7.5) -- (4,10);
    \filldraw[black!70] (4,10) circle (3pt);
    \filldraw[LightBlue] (3,3) circle (3pt) node[anchor=west] {$V(x^1)$};
    \draw[LightBlue] (2,6) -- (3,3);
    \filldraw[LightBlue] (4,6.4) circle (3pt) node[anchor=west] {$V(x^2)$};
    \draw[LightBlue,dashed] (3,3) -- (4,6.4);
    \filldraw[LightBlue] (5,8.5) circle (3pt) node[anchor=west] {$V(x^3)$};
    \draw[LightBlue,dashed] (4,6.4) -- (5,8.5);
    \filldraw[LightBlue] (6,10.5) circle (3pt) node[anchor=west] {$V(x^4)$};
    \draw[LightBlue,dashed] (5,8.5) -- (6,10.5);
\end{tikzpicture}
\caption{Optimization at ${k=2}$ \hphantom{platzhalter}}
\label{fig:expl_new_mpcimpl2}
\end{subfigure}
\begin{subfigure}[b]{0.242\textwidth}
\begin{tikzpicture}[scale=0.485][x=0.5cm, y=0.5cm,domain=0:11,smooth]
   \draw [color=gray!30]  [step=5mm] (-0.3,-0.5) grid (8.5,12);
   \draw[->,thick] (0,0) -- (8,0) node[right] {$k$};
   \draw[->,thick] (0,0) -- (0,11) node[above] {$V$};
   \foreach \c in {1,2,...,7}{
     \draw (\c,-.1) -- (\c,.1) ;
   }
    \filldraw[gray!40] (0,9) circle (8pt);
    \filldraw[black!70] (0,9) circle (3pt);
    \draw[black!50] (0,9) -- (1,8);
    \filldraw[black!70] (1,8) circle (3pt);
    \draw[black!50] (1,8) -- (2,6);
    \filldraw[gray!40] (2,6) circle (8pt);
    \filldraw[black!70] (2,6) circle (3pt);
    \draw[black!50, dashed] (2,6) -- (3,7.5);
    \filldraw[black!70] (3,7.5) circle (3pt);
    \draw[black!50, dashed] (3,7.5) -- (4,10);
    \filldraw[black!70] (4,10) circle (3pt);
    \filldraw[LightGreen] (3,3) circle (8pt);
    \filldraw[DarkGreen] (3,3) circle (3pt) node[anchor=west] {$V(x^0)$};
    \draw[black!50] (2,6) -- (3,3);
    \filldraw[black!70] (4,6.4) circle (3pt);
    \draw[black!50,dashed] (3,3) -- (4,6.4);
    \filldraw[black!70] (5,8.5) circle (3pt);
    \draw[black!50,dashed] (4,6.4) -- (5,8.5);
    \filldraw[black!70] (6,10.5) circle (3pt);
    \draw[black!50,dashed] (5,8.5) -- (6,10.5);
    \draw[DarkGreen] (3,3) -- (4,4);
    \filldraw[DarkGreen] (4,4) circle (3pt) node[anchor=west] {$V(x^1)$};
    \filldraw[DarkGreen] (5,5) circle (3pt) node[anchor=east] {$V(x^2)$};
    \draw[DarkGreen] (4,4) -- (5,5);
    \draw[DarkGreen] (5,5) -- (6,6);
    \filldraw[DarkGreen] (6,6) circle (3pt) node[anchor=west] {$V(x^3)$};
    \draw[DarkGreen] (6,6) -- (7,1);
    \filldraw[DarkGreen] (7,1) circle (3pt) node[anchor=east] {$V(x^4)$};
\end{tikzpicture}
\caption{Optimization at ${k=3}$ \hphantom{platzhalter}}
\label{fig:expl_new_mpcimpl3}
\end{subfigure}
\begin{subfigure}[b]{0.242\textwidth}
\begin{tikzpicture}[scale=0.485][x=0.5cm, y=0.5cm,domain=0:11,smooth]
   \draw [color=gray!30]  [step=5mm] (-0.3,-0.5) grid (8.5,12);
   \draw[->,thick] (0,0) -- (8,0) node[right] {$k$};
   \draw[->,thick] (0,0) -- (0,11) node[above] {$V$};
   \foreach \c in {1,2,...,7}{
     \draw (\c,-.1) -- (\c,.1) ;
   }
    \draw[yellow!70,line width=1.1mm] (0,9) -- (1,8);
    \draw[yellow!70,line width=1.1mm] (1,8) -- (2,6);
    \draw[yellow!70,line width=1.1mm] (2,6) -- (3,3);
    \draw[yellow!70,line width=1.1mm] (3,3) -- (4,4);
    \draw[yellow!70,line width=1.1mm] (4,4) -- (5,5);
    \draw[yellow!70,line width=1.1mm] (5,5) -- (6,6);
    \draw[yellow!70,line width=1.1mm] (6,6) -- (7,1);
    \filldraw[LightGreen] (0,9) circle (8pt);
    \filldraw[LightGreen] (2,6) circle (8pt);
    \filldraw[LightGreen] (3,3) circle (8pt);
    \filldraw[LightGreen] (7,1) circle (8pt);
    \filldraw[black] (0,9) circle (3pt) node[anchor=west] {$V(x(0))$};
    \draw[black] (0,9) -- (1,8);
    \filldraw[black] (1,8) circle (3pt) node[anchor=west] {$V(x(1))$};
    \draw[black] (1,8) -- (2,6);
    \filldraw[black] (2,6) circle (3pt) node[anchor=west] {$V(x(2))$};
    \filldraw[black] (3,3) circle (3pt) node[anchor=west] {$V(x(3))$};
    \draw[black] (2,6) -- (3,3);
    \draw[black] (3,3) -- (4,4);
    \filldraw[black] (4,4) circle (3pt) node[anchor=west] {$V(x(4))$};
    \filldraw[black] (5,5) circle (3pt) node[anchor=east] {$V(x(5))$};
    \draw[black] (4,4) -- (5,5);
    \draw[black] (5,5) -- (6,6);
    \filldraw[black] (6,6) circle (3pt) node[anchor=west] {$V(x(6))$};
    \draw[black] (6,6) -- (7,1);
    \filldraw[black] (7,1) circle (3pt) node[anchor=east] {$V(x(7))$};
\end{tikzpicture}
\caption{Lyapunov function values along actual states}
\label{fig:final_trajectory_expl_mpc_newimpl}
\end{subfigure}
\caption{Illustration of proposed MPC scheme: Consider Problem~\eqref{eq:display} with $m=4$ and $q=0$. The initial state is ${\color{DarkBlue}x^0}=x(0)$, whose Lyapunov function value is depicted in (a) and highlighted in green. After solving the finite-horizon optimal control problem, we obtain four predicted states and their corresponding Lyapunov function values ({\color{DarkBlue}$V(x^0),V(x^1),V(x^2),V(x^3),V(x^4)$}). Since there are multiple time indices for which the Lyapunov function decreases, we choose $\ld$ here as the index where the greatest descent occurs. We implement $\ld=2$ components of the control sequence and, consequently, declare {\color{DarkBlue}$x^2$} as the new initial state for the finite-horizon optimal control problem at time $k=2$. 
This problem is solved in (b) and we obtain again four states and their corresponding Lyapunov function values ({\color{LightBlue}$V(x^0),V(x^1),V(x^2),V(x^3),V(x^4)$}). We repeat the scheme and after solving Problem~\eqref{eq:display} three times, we obtain the trajectory of the closed-loop states and their Lyapunov function values shown in (d).}
\label{fig:expl_new_mpcimpl}
\end{figure*}
 
For further details, we refer the reader to our prior work in~\cite{AF-CE-BG:22,AF-CE-BG:23}. This paper is concerned with finding a systematic way to design the \hodclf~$V$, so that it can be utilized for \adc~inside of~\eqref{eq:display}. Similar to Lyapunov theory, in general settings, this is quite a challenging task; for this reason, we restrict ourselves to linear systems in what follows. 


\section{\hodclf s for linear systems}\label{sec:avenuelmi}
\noindent From now on, we consider linear discrete-time control systems
\begin{align}\label{eqn:linsys}
    x^{k+1} = Ax^k + Bu^k
\end{align}
with some initial state $x^0 \in \real^n$, where $x^k \in \real^n, \, u^k \in \real^p, \, A \in \real^{n,n}$ and $B \in \real^{n,p}$. Since we are now in the linear setting, let us focus on the quadratic \hodclf
\begin{equation*}
    V(x) = \|x\|^2.
\end{equation*}
Note that we choose the simplest quadratic function here, other quadratic functions would be possible as well.
In this case, \adc~reads:
\begin{align}\label{eq:adc}
&\frac{\sigma_m\|x^m\|^2 + \dots + \sigma_1\|x^1\|^2 }{m} - \|x^0\|^2 \leq -\alpha(x^0).
\end{align}
Suppose we utilize state feedback $u = Kx$ with some matrix $K \in \real^{p,n}$. Clearly, we have that 
\[
x^j = (A + BK)^j x^0,
\]
for all $ j \in \{1,\ldots m\}$. Therefore,
\[
\mathbf{x}= \Phi x^0,
\]
where $ \mathbf{x} =
((x^1)^T, \dots, (x^m)^T)^T
 $ and  
\begin{align*}
    \Phi &= \begin{bmatrix}
        A + BK\\
        (A + BK)^2\\
        \vdots\\
        (A + BK)^m
    \end{bmatrix}
\in \real^{nm,n}.
\end{align*}
By defining 
\begin{align*}
    \Lambda := \begin{bmatrix}
        \dfrac{\sigma_1}{m}I & & \\
         & \ddots & \\
         & & \dfrac{\sigma_m}{m}I
    \end{bmatrix} \in \real^{nm,nm},
\end{align*}

and also $\alpha(x^0) := \varepsilon \|x^0\|^2$, where $1>\varepsilon>0$, we can rewrite~\eqref{eq:adc} as a matrix inequality:
\begin{align*}
    &\hspace*{-3.6cm}\mathbf{x}^T\Lambda \mathbf{x} - (x^0)^Tx^0\nonumber\\
    =\, (\Phi x^0)^T \Lambda \Phi x^0 - (x^0)^Tx^0&\leq -\alpha(x^0)= -\varepsilon\|x^0\|^2. \nonumber
\intertext{By rearranging further, we obtain}
    -(x^0)^T \Phi^T\Lambda \Phi x^0 + (x^0)^TIx^0&\geq \varepsilon(x^0)^Tx^0 \quad \forall x^0 \in \real^n \nonumber \\
   \Leftrightarrow -\Phi^T\Lambda \Phi + I & \succeq \varepsilon I,
\end{align*}
which is a Linear Matrix Inequality (LMI) with the decision variable $\Lambda$. If we pose no constraint on this LMI, the solution will just be the zero matrix, as $I \succeq \varepsilon I$ for small $\varepsilon$.
We do, however, have two constraints on the weights: First, they need to be non-negative; second, they need to satisfy~\eqref{eqn:sigma_ave}. Altogether, our LMI consists of three parts:
\begin{align}
    -\Phi^T\Lambda \Phi + I -\varepsilon I &\succeq 0\label{eqn:fullLMI1}\\
    \Lambda &\succeq 0\label{eqn:fullLMI2}\\
    \tfrac{1}{m}(\sigma_1 + \dots + \sigma_m) &\geq 1 \text{ or } \trace(\Lambda)\geq n. \label{eqn:fullLMI3}
\end{align}
This leads us to the first  result of this paper.
\begin{theorem}\label{thm:mainthm}
Consider system~\eqref{eqn:linsys} and a stabilizing matrix $K \in \mathbb{R}^{p,n}$, i.e. the absolute value of each eigenvalue of $A+BK$ is strictly less than one. Then there exists a sufficiently large $m \in \mathbb{N}_{>0}$ such that $\sigma_1, \dots, \sigma_m \in \mathbb{R}_{\geq 0}$ satisfy the LMI described by~\eqref{eqn:fullLMI1},~\eqref{eqn:fullLMI2},~\eqref{eqn:fullLMI3} with $0 < \varepsilon < 1$.    
\end{theorem}
\begin{proof}
Suppose  $K \in \mathbb{R}^{p,n}$ satisfies $\rho(A+BK)<1$. Our goal is to find a solution to the LMI~\eqref{eqn:fullLMI1},~\eqref{eqn:fullLMI2},~\eqref{eqn:fullLMI3}. Let us first consider the setting where $\sigma_1, \dots, \sigma_{m-1}=0$ and $\sigma_m = m$. With this choice,~\eqref{eqn:fullLMI2} and~\eqref{eqn:fullLMI3} are automatically satisfied. We show that we can find $m \in \mathbb{N}_{>0}$ such that the average descent property described by~\eqref{eqn:fullLMI1} is satisfied. With the above choices,~\eqref{eqn:fullLMI1} reads
\begin{align}\label{eqn:fullLMI1_subbedin}
    -((A+BK)^m)^T(A+BK)^m + (1-\varepsilon)I \succeq 0.
\end{align}
Note that
\begin{align*}
    \, \|((A+BK)^m)^T&(A+BK)^m\| \\
    \leq & \, \|((A+BK)^m)^T\| \|(A+BK)^m\| \\ 
    = \, & \|(A+BK)^m\| \|(A+BK)^m\|.
\end{align*}
Since $\rho(A+BK)<1$, it is a well-known fact~\cite{EI-HK:94} that $\lim_{m \to \infty} \|(A+BK)^m\|=0$; therefore, 
\begin{align*}
\lim_{m \to \infty} ((A+BK)^m)^T(A+BK)^m &= 0.
\end{align*} 
Let us denote the eigenvalues of ${((A+BK)^m)^T(A+BK)^m}$ by $\{\lambda_{m,1}, \dots, \lambda_{m,n}\}$.
This in turn implies that each of the eigenvalues $\lambda_{m,i}$ converges to the eigenvalues of the zero matrix, which are all zero of course, as $m$ goes to infinity. Let us now choose $\tilde \varepsilon$ with $0 < \tilde \varepsilon \leq 1 - \varepsilon$. Then there exists $\tilde m \in \mathbb{N}_{>0}$ such that for all eigenvalues $\lambda_{\tilde m,i}$ we have $|\lambda_{\tilde m,i}| < \tilde \varepsilon$. Since the eigenvalues of  
\begin{equation}\label{eqn:A+BK1-Eps}
    -((A+BK)^{\tilde m})^T(A+BK)^{\tilde m} + (1-\varepsilon)I
\end{equation}
are equal to $(1-\varepsilon) - \lambda_{\tilde m,i}$, we have that 
\begin{align*}
   (1-\varepsilon) -\lambda_{\tilde m,i} & \geq (1-\varepsilon) - |\lambda_{\tilde m,i}| \\
   &> (1-\varepsilon) - \tilde \varepsilon \geq 0.
\end{align*}
With all the eigenvalues of the matrix in~\eqref{eqn:A+BK1-Eps} being non-negative, we conclude that the matrix itself is positive (semi) definite. In other words, we have found $\tilde m \in \mathbb{N}_{>0}$ such that~\eqref{eqn:fullLMI1_subbedin} is satisfied.
In the more general case, where all the weights $\sigma_i$ are positive, we have $\tilde m-1$ more matrices to consider in~\eqref{eqn:fullLMI1}. Clearly, we can find small enough weights $\sigma_1, \dots, \sigma_{\tilde m-1}$ such that $-\Phi^T\Lambda \Phi$ is still dominated by $\frac{\sigma_{\tilde m}}{\tilde m}(-((A+BK)^{\tilde m})^T(A+BK)^{\tilde m})$ and hence, stays positive definite.  
This finishes the proof. 
\end{proof}

\subsection{Numerical Example}\label{sec:simulations}
To illustrate our results, we will use the introduced LMI approach for an example.  In particular, we consider the following problem: 
\begin{problem}\label{problem:simulation1}
\begin{align*}
\min &\sum_{j = 0}^{N_p-1}\|x^j\|_1 + 5\|u^j\|^2\\
\mathrm{s.t.} &\,x^{j+1} = \begin{bmatrix}
        2.13 & 1 & 1\\
        0 & 1 & 0.3 \\
        0 & 0 & 0.5
    \end{bmatrix}x^j + \begin{bmatrix}
        0 \\0 \\1
    \end{bmatrix}u^j\\
    &\,x^0=x\\
&\,\frac{1}{N_p}(\sigma_{N_p}\|x^{N_p}\|^2+\dots + \sigma_1\|x^1\|^2) - \|x^0\|^2 \leq -\varepsilon \|x^0\|^2. 
\end{align*}
\end{problem}
The initial state is $[4 \,\, 12 \,\, 15]^T$ and the prediction horizon is chosen to be $N_p=10$. Note that this problem has a non-quadratic running cost, which means that it is not directly clear which terminal cost should be utilized in the standard MPC.
To apply our methodology, we first select a stabilizing feedback $K$ for the control system in Problem~\ref{problem:simulation1}
\begin{equation*}
    K = \begin{bmatrix}-3.5507 &-2.6749 &   -2.4633\end{bmatrix}.
\end{equation*}
Next, we find a sufficiently large $m \in \mathbb{N}_{>0}$ such that the LMI~\eqref{eqn:fullLMI1},~\eqref{eqn:fullLMI2},~\eqref{eqn:fullLMI3} admits a solution; note that the existence of such an $ m $ is guaranteed by Theorem~\ref{thm:mainthm}.
Using numerical evaluations\footnote{The numerical evaluations consist of solving the LMI~\eqref{eqn:fullLMI1},~\eqref{eqn:fullLMI2},~\eqref{eqn:fullLMI3} through MATLAB and the {\tt sedumi} package.}, one can observe that selecting $m \geq 6$ suffices. We have chosen $m=10$ and have solved the LMI~\eqref{eqn:fullLMI1},~\eqref{eqn:fullLMI2},~\eqref{eqn:fullLMI3} to obtain:
\begin{align*}
&\sigma_1 = 0.0055, \sigma_2 = 0.0524, \sigma_3 = 0.0660, \sigma_4 = 0.0655,\\
&\sigma_5 = 0.0762, \sigma_6 = 0.0952, \sigma_7 = 0.1201, \sigma_8 =  0.1479, \\
&\sigma_9 = 0.1745, \sigma_{10} = 0.1967.
\end{align*}
With the weights for \adc~at hand, we solve Problem~\ref{problem:simulation1} with the flexible-step MPC scheme given in Algorithm~\ref{algo:ho-dclf} by using {\tt fmincon} from MATLAB.
In step four of Algorithm~\ref{algo:ho-dclf}, we find the index $\ell_{decr}\in \{1,\dots,N_p\}$ with the \emph{maximal} descent for the \hodclf~$V(x) = \|x\|^2$. We then choose $\varepsilon = 10^{-10}$ for the right-hand side of \adc. Figure~\ref{fig:state-trajectories-ex1} shows the state trajectories according to the solution of Problem~\ref{problem:simulation1}. After a transient phase, the state is successfully stabilized to the origin. A summary of the implemented steps at each optimization instance of Problem~\ref{problem:simulation1} is given in Figure~\ref{fig:impl-steps-ex1}, where notably, a flexible number of steps is implemented in each optimization instance. 

\begin{figure}
    \centering
    \includegraphics[width=0.7\linewidth]{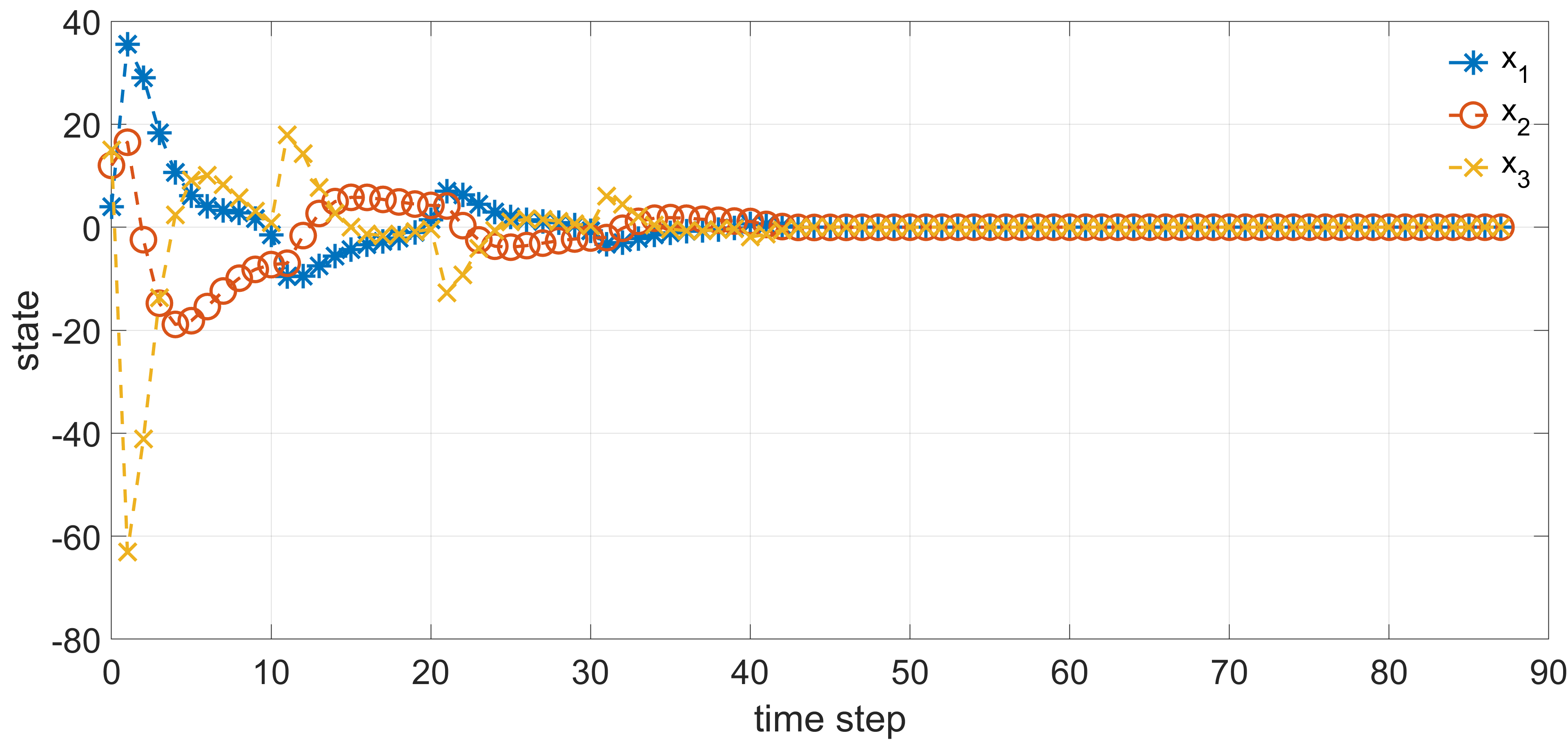}
    \caption{State trajectories according to the solution of Problem~\ref{problem:simulation1}}
    \label{fig:state-trajectories-ex1}
\end{figure}

\begin{figure}
    \centering
    \includegraphics[width=0.3\linewidth]{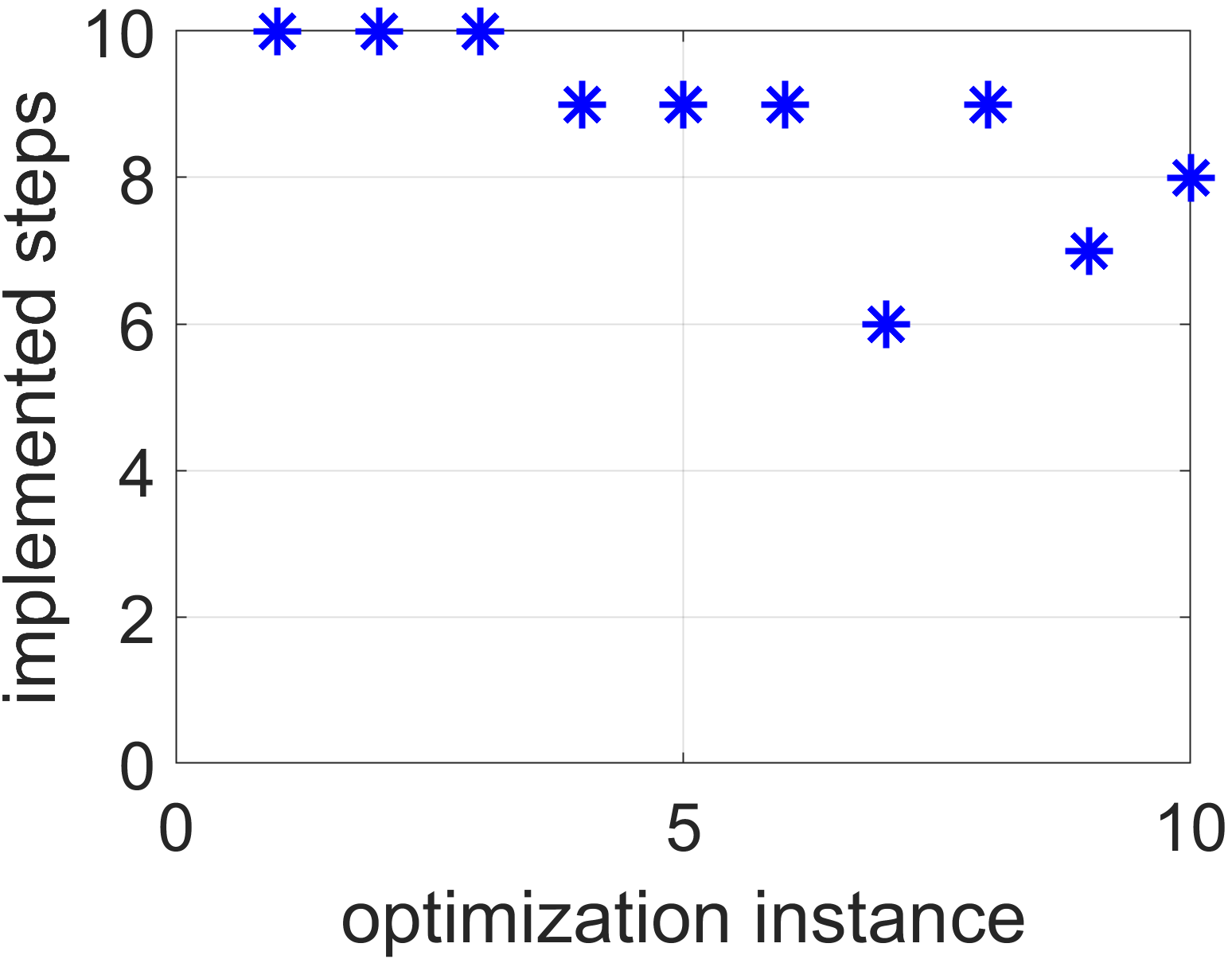}
    \caption{In each optimization instance, Problem~\ref{problem:simulation1} is solved. The solution yields the number of implemented steps, depicted on the vertical axis. Note that the sum of implemented steps of the optimization instances corresponds to the time step of the implementation. }
    \label{fig:impl-steps-ex1}
\end{figure}


\section{Quadratic Common \hodclf s}\label{sec:commonLF}
We now demonstrate how the  
introduced LMI approach
can be beneficial in the context of stabilization of switched linear systems. Let us start by recalling some terminology~\cite{DL:03}. 
We consider a family of nonlinear control systems described by
\begin{align}\label{eqn:general_nonlin_family_sys}
    x^{k+1} = f_{\uptheta}(x^k, u^k), \quad \uptheta \in \Uptheta,
\end{align}
where $\Uptheta$ is some index set and $f_{\uptheta}: \mathbb{R}^n \times \mathbb{R}^p \to \mathbb{R}^n$ is continuous for each $\uptheta \in \Uptheta$. 
With abuse of notation, we denote the switching signal by $\uptheta : \mathbb{N} \to \Uptheta$, which specifies, at each time instant $k \in \mathbb{N}$, the index $\uptheta(k) \in \Uptheta$ of the active subsystem, i.e. the system from the family~\eqref{eqn:general_nonlin_family_sys} that is currently being followed. 
Hence, a switched system with time-dependent switching can be described by the equation
\begin{equation}\label{eqn:nonlin_switched_system}
x^{k+1} = f_{\uptheta(k)}(x^k,u^k).
\end{equation}
In what follows, we discuss the asymptotic stability of the switched system~\eqref{eqn:nonlin_switched_system} under arbitrary switching. Let us precisely define what that entails. 
\begin{definition}
A switched system~\eqref{eqn:nonlin_switched_system} is uniformly asymptotically stable if for all $\varepsilon>0$ there exists a $\delta>0$ such that for all switching signals $\uptheta \in \Uptheta$ the trajectories of~\eqref{eqn:nonlin_switched_system} with $0 < \|x^0\| < \delta$ satisfy $\|x^k\| < \varepsilon$.
\end{definition}
We assume that the individual subsystems have the origin as the common equilibrium point: $f_{\uptheta}(0,0) = 0$ for all $\uptheta \in \Uptheta$. 
A useful tool when investigating the stability of a \emph{family} of control systems is the notion of \emph{common} Lyapunov functions, which we define next.
\begin{definition}
A Lyapunov function $V: \mathbb{R}^n \to \mathbb{R}_{\geq 0}$ is called common with respect to~\eqref{eqn:general_nonlin_family_sys} if it is a continuous, radially unbounded, positive definite function and satisfies the Lyapunov condition, i.e. there exists a continuous, positive definite function $W : \mathbb{R}^n \to \mathbb{R}_{\geq 0}$ such that for all $x\in \mathbb{R}^n$ and all $\uptheta \in \Uptheta$ there exists $u \in \mathbb{R}^p$ with
\begin{equation*}
    V(f_{\uptheta}(x,u)) - V(x) \leq -W(x).
\end{equation*}
\end{definition}
In other words, the single Lyapunov function $V$ decreases along trajectories of all subsystems of~\eqref{eqn:general_nonlin_family_sys}.
For simplicity, we assume in what follows that the index set $\Uptheta$ is finite. Then the existence of a common Lyapunov function abiding the slightly weaker condition  
\begin{equation}\label{eqn:Lyap_conditionw}
    V(f_{\uptheta}(x,u)) - V(x) < 0
\end{equation}
for all $x \in \mathbb{R}^n\setminus\{0\}$ and all $\uptheta \in \Uptheta$
implies global uniform (w.r.t. $\uptheta$) asymptotic stability, compare with the continuous time result~\cite[Theorem~2.1]{DL:03}.\footnote{
More generally, condition~\eqref{eqn:Lyap_conditionw} is sufficient for global uniform asymptotic stability under a compactness assumption on the index set $\Uptheta$ and a continuity assumption on $f_{\uptheta}$ with respect to $\uptheta$. Both of these assumptions trivially hold when $\Uptheta$ is finite~\cite{DL:03}.} 

Similar to the common Lyapunov function, a common \hodclf~for~\eqref{eqn:general_nonlin_family_sys} can be defined and utilized to show global uniform (w.r.t. $\uptheta$) asymptotic stability. Essentially, it must hold that for all $\uptheta \in \Uptheta$ there exists a feasible control such that \adc~is satisfied.

In the context of linear systems, a natural question is whether it is sufficient to work with \emph{quadratic} common Lyapunov functions, i.e. Lyapunov functions $V: \mathbb{R}^n \to \mathbb{R}_{\geq 0}$ in the form of $V(x) = x^TPx$, where $P \in \mathbb{R}^{n,n}$ is positive definite.
The next main result addresses precisely this question, with the first part being known for switched dynamical systems or continuous-time control systems~\cite{DL:03,JH-DL-DA-ES:05}. 

\begin{theorem}\label{prop:2lemmata}
The switched system
\begin{equation}\label{eqn:switchedlinsystem}
x^{k+1} = x^k + \uptheta\,h
\begin{bmatrix}
    0 & 1\\
    -(1+\alpha) & 0 \end{bmatrix} x^k
 + h \begin{bmatrix} 0 \\ 1\end{bmatrix} u^k, 
\end{equation}
where $\uptheta \in \{-1,1\}, h>0$ and $\alpha=0.5$,
does not admit a quadratic common control Lyapunov function. It does, however, admit a quadratic common \hodclf~for $h=0.1$. 
\end{theorem}
We prove this result through two lemmata.
\begin{lemma}\label{thm:nocommonclLF}
The switched system~\eqref{eqn:switchedlinsystem} does not admit a quadratic common control Lyapunov function.
\end{lemma}

\begin{proof}
The linear system~\eqref{eqn:switchedlinsystem} can be written more compactly as
\begin{align}
    x^{k+1} &= A_{\uptheta}x^k + Bu^k, \quad \uptheta \in \{-1,1\}\label{eqn:systemA1B1}
    \intertext{with }
    A_1 &= \begin{bmatrix}
        1 & h\\
        -1.5h & 1
    \end{bmatrix}, \,
    A_{-1} = \begin{bmatrix}
        1 & -h\\
        1.5h & 1
    \end{bmatrix}, \,
    B = h\begin{bmatrix}
        0 \\ 1
    \end{bmatrix}.\nonumber
\end{align}
\noindent For later use, we also introduce the matrices
\begin{align*}
\mathcal{A}_1 &= \begin{bmatrix}
        0 &1\\
        -1.5 & 0
    \end{bmatrix}, \, 
\mathcal{A}_{-1} = \begin{bmatrix}
        0 & -1\\
        1.5 & 0
    \end{bmatrix}, \,
\mathcal{B}= \begin{bmatrix}
        0\\1
    \end{bmatrix}.
\end{align*}
By way of contradiction, let us assume that a quadratic common control Lyapunov function $V(x)=x^TPx$, with a positive definite matrix $P \in \mathbb{R}^{n,n}$, exists. Without loss of generality, the matrix $P$ has the form 
\begin{align*}
    P = \begin{bmatrix}
        1 & q\\ q & r
    \end{bmatrix}.
\end{align*}
Since $\mathcal{B} = [ 0, \,\, 1]^T$, we have that $\mathcal{B}^{T}Py = 0$ for $y = [-r, \,\, q]^T$.
The function $V$ needs to satisfy the descent property~\eqref{eqn:Lyap_conditionw}
\begin{align*}
    (A_{\uptheta}x + Bu)^TP(A_{\uptheta}x + Bu) - x^TPx &<0
 \end{align*}   
 with one common $P$ for both $\uptheta = -1$ and $\uptheta = 1$.
 Expanding the above inequality yields \begin{align*}x^TA_{\uptheta}^TPA_{\uptheta}x + 2uB^TPA_{\uptheta}x + uB^TPBu - x^TPx &<0.
\end{align*}
We rewrite the above inequality by using $A_{\uptheta} = I+h\mathcal{A_{\uptheta}}$ and $B = h\mathcal{B}$
\begin{align*}
x^T(I+h\mathcal{A_{\uptheta}})^TP(I+h\mathcal{A_{\uptheta}})x + 2hu\mathcal{B}^{ T}P(I+h\mathcal{A_{\uptheta}})x 
    + h^2u\mathcal{B}^{T}P\mathcal{B}u - x^TPx &< 0 \\
\Rightarrow \, x^T(P + h\mathcal{A_{\uptheta}}^{T}P + hP\mathcal{A_{\uptheta}} + h^2\mathcal{A_{\uptheta}}^{T}P\mathcal{A_{\uptheta}} - P)x 
    + 2hu\mathcal{B}^{T}P(I + h\mathcal{A_{\uptheta}})x +  h^2u\mathcal{B}^{T}P\mathcal{B}u &< 0\\
\Rightarrow   \, hx^T(\mathcal{A_{\uptheta}}^{T}P + P\mathcal{A_{\uptheta}})x + h^2x^T\mathcal{A_{\uptheta}}^{T}P\mathcal{A_{\uptheta}}x 
    + 2hu\mathcal{B}^{T}P(I + h\mathcal{A}_{\uptheta})x +  h^2u\mathcal{B}^{T}P\mathcal{B}u &< 0.
\end{align*}
By dividing by $h>0$, this is equivalent to
\begin{align*}
&x^T(\mathcal{A}_{\uptheta}^{T}P + P\mathcal{A}_{\uptheta})x + hx^T\mathcal{A}_{\uptheta}^{T}P\mathcal{A}_{\uptheta}x 
     + 2u\mathcal{B}^{T}P(I + h\mathcal{A}_{\uptheta})x + hu\mathcal{B}^{T}P\mathcal{B}u<0.
\end{align*}
Written in two equations with $\uptheta = +1 $ and $\uptheta = -1$, this reads
\begin{align*}
x^T(\mathcal{A}_{1}^TP + P\mathcal{A}_{1})x + hx^T\mathcal{A}^T_{1} P \mathcal{A}_{1} x + 2u \mathcal{B}^TP(I + h\mathcal{A}_{1})x + hu \mathcal{B}^TP\mathcal{B}u &< 0 \\
x^T(\mathcal{A}_{-1}^TP + P\mathcal{A}_{-1})x + hx^T\mathcal{A}^T_{-1} P \mathcal{A}_{-1} x + 2u \mathcal{B}^TP(I + h\mathcal{A}_{-1})x + hu \mathcal{B}^TP\mathcal{B}u &< 0.
\end{align*}
If both of these expressions are negative, then their sum must also be negative, i.e.
\begin{align*}
&x^T(\mathcal{A}_{1}^TP + P\mathcal{A}_{1})x + hx^T\mathcal{A}^T_{1} P \mathcal{A}_{1}x + 2u \mathcal{B}^TP(I + h\mathcal{A}_{1})x + \\
&\quad hu \mathcal{B}^TP\mathcal{B}u + 
x^T(\mathcal{A}_{-1}^TP + P\mathcal{A}_{-1})x + hx^T\mathcal{A}^T_{-1} P \mathcal{A}_{-1} x +\\
& \quad 2u \mathcal{B}^TP(I + h\mathcal{A}_{-1})x + hu \mathcal{B}^TP\mathcal{B}u < 0.
\intertext{Since $\mathcal{A}_{-1} = -\mathcal{A}_{1} $, it holds}
&x^T(\mathcal{A}_{1}^TP + P\mathcal{A}_{1})x + hx^T\mathcal{A}^T_{1} P \mathcal{A}_{1} x + 2u \mathcal{B}^TPx + \\
& \quad 2hu \mathcal{B}^TP\mathcal{A}_{1}x + hu \mathcal{B}^TP\mathcal{B}u  
-x^T(\mathcal{A}_{1}^TP + P\mathcal{A}_{1})x + \\
& \quad hx^T\mathcal{A}^T_{1} P \mathcal{A}_{1} x + 2u \mathcal{B}^TPx - 2hu \mathcal{B}^TP\mathcal{A}_{1}x +hu \mathcal{B}^TP\mathcal{B}u < 0.
\intertext{After cancellations, we obtain}
&2hx^T\mathcal{A}^T_{1} P \mathcal{A}_{1} x + 4u \mathcal{B}^TPx + 2hu \mathcal{B}^TP\mathcal{B}u < 0.
\end{align*}
If we now sub in $y = [-r, \, \, q]^T$ for $x$, the term $4u \mathcal{B}^TPy$ equals zero. This yields 
\begin{align*}
2hy^T\mathcal{A}^T_{1} P \mathcal{A}_{1} y + 2hu \mathcal{B}^TP\mathcal{B}u < 0.
\end{align*}
Since $P$ is positive definite, this is a contradictory statement.
To summarize, the existence of a quadratic common control Lyapunov function has led us to a contradiction, which means that there cannot exist such a Lyapunov function.
\end{proof}
\begin{lemma}\label{lem:commonhodclf}
The switched system~\eqref{eqn:switchedlinsystem} admits a quadratic common \hodclf~for $h=0.1$.    
\end{lemma}
\begin{proof}
It can be readily verified that for $m=10$ the weights\footnote{The weights can be found through MATLAB and the {\tt sedumi} package  by solving the coupled LMI 
\begin{align*}
    -\Phi_1^T \Lambda \Phi_{1} + I - \varepsilon I &\succeq 0\\
    -\Phi_{-1}^T \Lambda \Phi_{-1} + I - \varepsilon I &\succeq 0\\
    \Lambda &\succeq 0\\
    \trace(\Lambda)&\geq n = 2.
\end{align*}
}
\begin{align}\label{eqn:weights-switched-sys}
&\sigma_1 = 0.0644, \sigma_2 = 0.0570, \sigma_3 = 0.0589, \sigma_4 = 0.0655,\nonumber\\
&\sigma_5 = 0.0775, \sigma_6 = 0.0959, \sigma_7 = 0.1227, \sigma_8 = 0.1646,\\
&\sigma_9 = 0.2488, \sigma_{10} = 0.5447.\nonumber
\end{align} 
satisfy~\eqref{eqn:fullLMI1},~\eqref{eqn:fullLMI2},~\eqref{eqn:fullLMI3} simultaneously for system~\eqref{eqn:switchedlinsystem} with $\uptheta =1$ and system~\eqref{eqn:switchedlinsystem} with $\uptheta=-1$. This can be seen by using the stabilizing matrices 
\begin{align*}
K_1 =  [-5.4017 \,\,\, -7.0985], \,\, K_{-1} =  [5.4017 \,\,\,  -7.0985]    
\end{align*}
for~\eqref{eqn:switchedlinsystem} with $\uptheta =1$ and $\uptheta =-1$ and building $\Phi_1$ and $\Phi_{-1}$ consisting of powers of $(A_1 + BK_1)$ and $(A_{-1} + BK_{-1})$, respectively.
Since the above weights satisfy~\eqref{eqn:fullLMI1},~\eqref{eqn:fullLMI2},~\eqref{eqn:fullLMI3} for the switched system, i.e. system~\eqref{eqn:switchedlinsystem} with $\uptheta = 1$ and $\uptheta = -1$, it admits a common \hodclf, namely $V(x) = \|x\|^2$. 
\end{proof}

We now illustrate how the quadratic common \hodclf~facilitates stabilization by solving the optimal control Problem~\ref{problem:simulation} within the flexible-step MPC scheme (Algorithm~\ref{algo:ho-dclf}). The weights for \adc~are given by~\eqref{eqn:weights-switched-sys}, the initial state is $\begin{bmatrix}
    4 &5
\end{bmatrix}^T$, the prediction horizon is $N_p=10$ and we utilize the solver {\tt fmincon} from MATLAB.  

\begin{problem}\label{problem:simulation}
\begin{align*}
&\min \sum_{j = 0}^{N_p-1}\|x^j\|^2 + 5\|u^j\|^2\\
& \, \mathrm{s.t.} \,x^{j+1} = x^j + h \begin{bmatrix} 0 \\ 1\end{bmatrix} u^j + \left\{\begin{array}{l}
\hspace*{-0.2cm}h\begin{bmatrix}
    0 & 1\\
    -0.5 & 0 \end{bmatrix} x^j,
 \, \text{if } k+j \ \mathrm{even}\\
 \hspace*{-0.2cm}- h
\begin{bmatrix}
    0 & 1\\
    -0.5 & 0 \end{bmatrix} x^j \hfill \mathrm{otherwise}
\end{array}\right.\\
&\hspace{0.8cm}\,x^0=x\\
&\hspace{0.8cm}\,\frac{1}{N_p}(\sigma_{N_p}\|x^{N_p}\|^2+\dots + \sigma_1\|x^1\|^2) - \|x^0\|^2 \leq \hspace*{-0.1cm} -\varepsilon \|x^0\|^2 
\end{align*}
\end{problem}
Here, $k$ is the time step of the closed-loop system. In step four of Algorithm~\ref{algo:ho-dclf}, we find the \emph{first} index $\ell_{decr}\in \{1,\dots,N_p\}$ which achieves a descent for the \hodclf~$V(x) = \|x\|^2$. To achieve a descent, we choose $\varepsilon = 10^{-5}$ for the right-hand side of \adc. The state trajectories of the closed-loop system are displayed in Figure~\ref{fig:states_flexMPC}. It is evident, that after a transient phase, the closed-loop system was successfully stabilized to the origin. In Figure~\ref{fig:implsteps_flexMPC}, the number of implemented steps in each optimization instance is displayed. We see that in multiple optimization instances, the solver makes use of more than one step of the computed control sequence.

It is worth pointing out that our scheme works for any arbitrary switching pattern (which must be known for the next $ m $ time steps so that our scheme can be implemented).

\begin{figure}
    \centering
    \includegraphics[width=0.7\linewidth]{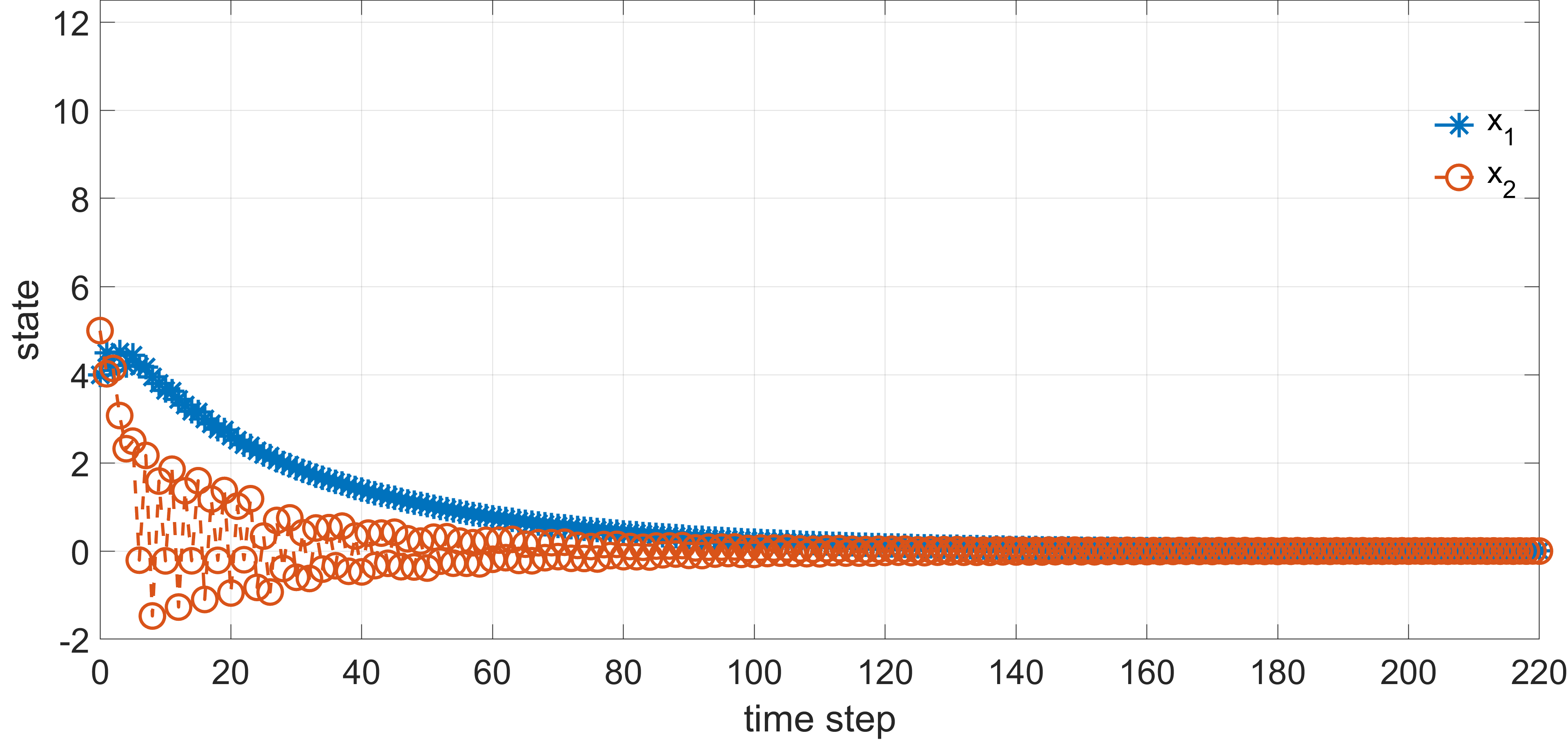}
    \caption{State trajectories according to the solution of Problem~\ref{problem:simulation}}
    \label{fig:states_flexMPC}
\end{figure}
\begin{figure}
    \centering
    \includegraphics[width=0.7\linewidth]{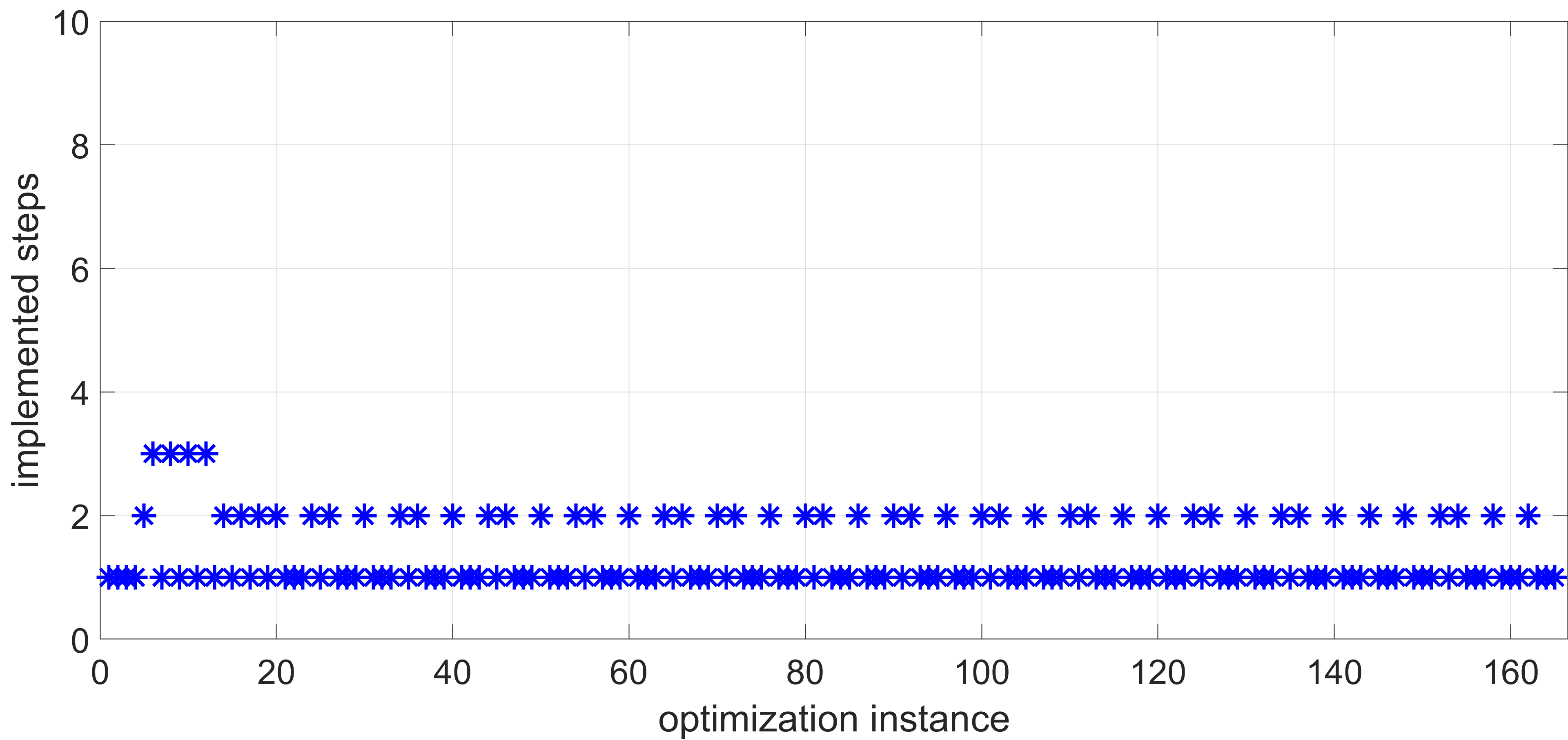}
    \caption{In each optimization instance, Problem~\ref{problem:simulation} is solved. The solution yields the number of implemented steps, depicted on the vertical axis. Note that the sum of implemented steps of the optimization instances corresponds to the time step of the implementation. }
    \label{fig:implsteps_flexMPC}
\end{figure}
\begin{figure}
    \centering
    \includegraphics[width=0.7\linewidth]{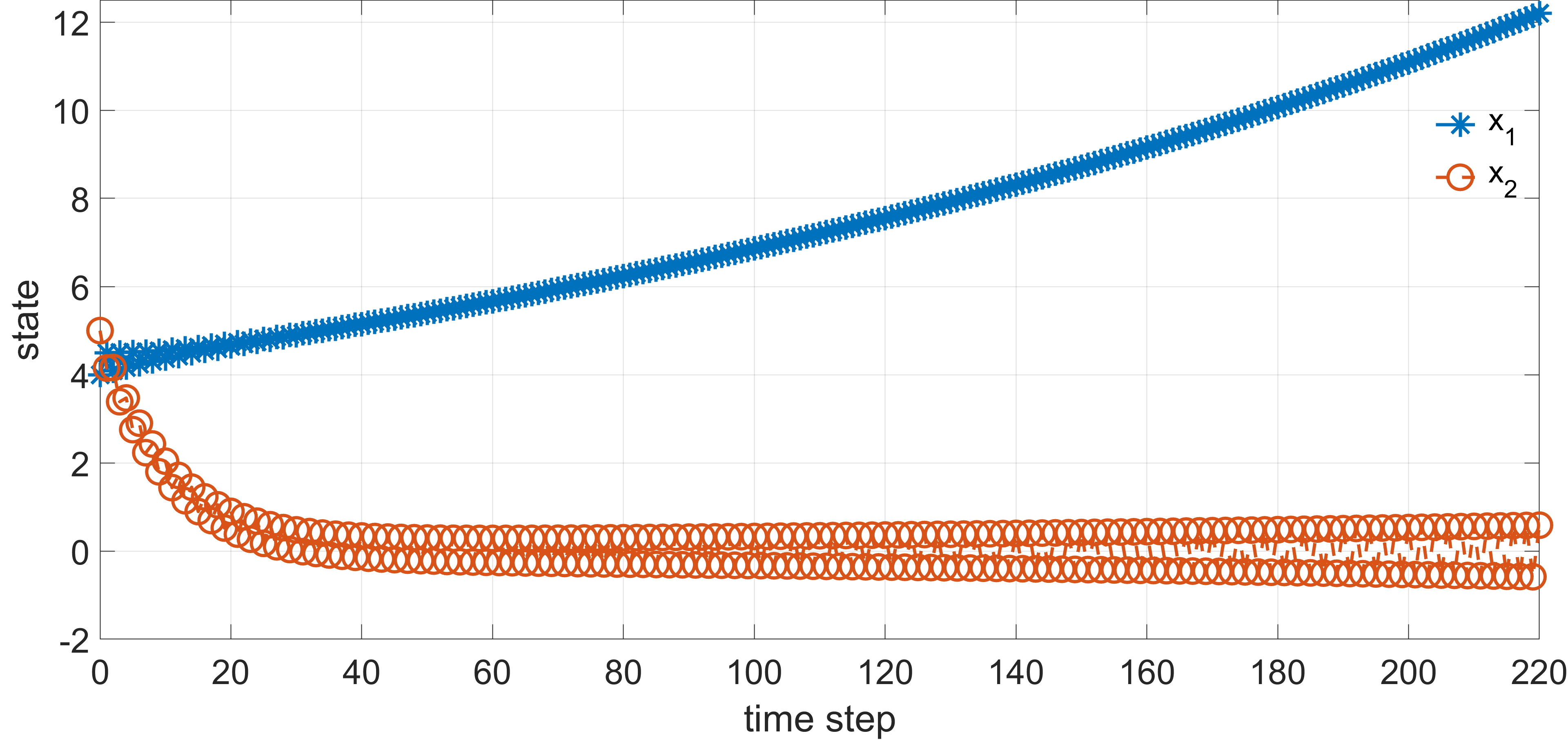}
    \caption{State trajectories according to standard MPC with the terminal cost $\phi(x^{N_p})=480(x^{N_p})^Tx^{N_p}$}
    \label{fig:states-stdMPC-480}
\end{figure}

We compare the results from flexible-step MPC with stan-
dard MPC. In standard MPC, the cost
\begin{align*}
    \sum_{j=0}^{N_p-1} \|x^j\|^2 + 5\|u^j\|^2 + \phi(x^{N_p})
\end{align*}
is optimized in each iteration. In other words, the stabilizing ingredient
of a quadratic terminal cost $\phi: \mathbb{R}^n \to \mathbb{R}_{\geq 0}$ is used.
As it is unclear how to choose the terminal cost systematically, we have tried three different terminal costs, i.e. $    \phi(x^{N_p}) = (x^{N_p})^TP_1x^{N_p}, 
    \phi(x^{N_p}) = (x^{N_p})^TP_{-1}x^{N_p}$ 
and $\phi(x^{N_p}) = 480(x^{N_p})^Tx^{N_p}$,
where $P_1$ and $P_{-1}$ are the solutions of the Riccati equation for the $\uptheta = 1$ and $\uptheta = -1$ dynamics~\eqref{eqn:switchedlinsystem}, respectively.
As expected, standard MPC fails to stabilize the state with all of these terminal costs. This is depicted for one of the terminal costs, e.g. $\phi(x^{N_p}) = 480(x^{N_p})^Tx^{N_p}$, in Figure~\ref{fig:states-stdMPC-480}, where the state trajectories are shown according to standard MPC.  


\section{Conclusion}\label{sec:conclusion}
In this paper, we showed that the average decrease constraint of the recently introduced framework of flexible-step MPC can be verified using an LMI in the case of linear systems. As a key consequence, we were able to find a quadratic common \emph{generalized} Lyapunov function for a  
switched linear control system, where the existence of quadratic common Lyapunov functions is not guaranteed. Using this, we showcased an exemplary optimal control problem with a switched system, where flexible-step MPC can overcome challenges faced in standard MPC. 
In the future, we will explore applications to robotics and power systems, where switched systems naturally appear and create challenges that we believe can be overcome by our method.


\end{document}